\documentclass[12pt]{amsart} 

\usepackage{amssymb}
\usepackage{amsfonts}
\usepackage[dvips]{graphics}
\usepackage{amsmath}
\usepackage{amsthm}
 
\newtheorem{thm}{Theorem}[section]
\newtheorem{crl}[thm]{Corolary}
\newtheorem{lema}[thm]{Lemma}

\newtheorem{rmk}[thm]{Remark}

\newcommand{\F}{\mathbb{F}}

\title{Left ideals in matrix rings over finite fields}

\author{R. A. Ferraz}

\address{Instituto de Matem\'atica e Estat\'istica, Universidade de S\~ao Paulo, Caixa Postal 66281, CEP-05315-970, S\~ao Paulo,
 Brazil.}  \email{raul@ime.usp.br}

\author{C. Polcino Milies}

\address{Instituto de Matem\'atica e Estat\'istica, Universidade de S\~ao Paulo, Caixa Postal 66281, CEP-05315-970, and Universidade Federal do ABC, CEP 09210-580, S\~ao Paulo,
 Brazil.}  \email{polcino@ime.usp.br and polcino@ufabc.edu.br}

\author{E. Taufer}
\address{Instituto de Matem\'atica, Estat\'istica e F\'isica, Universidade Federal do Rio Grande, Campus Carreros. Av. It\'alia km 8, CEP 96201-900 RS, Brazil}\email{editetaufer@furg.br}

\begin{document}
% Nome das 'captions'
 
\thanks{
The second author was partially supported by CNPq., Proc. 300243/79-0(RN) and FAPESP, Proc 2015/09162-9.}

\begin{abstract}
It is well-known that each left ideals in a matrix rings over a finite field is generated by an idempotent matrix. In this work we compute the number of left ideals in these rings, the number of different idempotents generating each left ideal, and give explicitly a set of idempotent generators of all left ideals of a given rank.   \end{abstract}

\maketitle

\section{Introduction}

Let $\F_q$ be the field with $q$ elements and denote by $M_n(\F_q)$ full ring of $n \times n$ matrices over $\F_q$. Given $M \in M_n(\F_q)$, the subspace of $\F_q^n$ generated by the rows of $M$, whose dimension is precisely the rank of  $M$, will be denoted by $W(M)$. Since $ M_n (\F_q)$ is semi-simple,  every left ideal is generated by just one element. In particular, every left ideal has a generator that is an idempotent. 

The left ideal generated by a matrix $M$ will be denoted by $\langle M \rangle $.
Notice that, since $rank(AM) \leq rank(M)$ for all $A\in M_n(\F_q)$ we have that
$$rank(M) = sup\{rank(XM) \;|\; X\in M_n(\F_q)\}.$$

Consequently, two generators of the same ideal  are of the same rank. For an ideal $J = \langle M \rangle $ we define $rank(J)=rank(M)$.

Results obtained here refer to left ideals, but similar results hold for right ideals.

\section{Ideals and Subspaces}

We intend to count the number of left ideals in $M_n(\F_q)$. To do so,  we will establish a correspondence of these ideals with subspaces of $\F_q^n$.
We begin with the following.

\begin{lema}  \label{combLinear}Given $ M_1, M_2 \in M_n(\F_q)$, then $ \langle M_1\rangle \subset \langle M_2 \rangle $  if, and only if, the rows of $M_1$ are linear combinations of the rows of $M_2$ .
\end{lema}
\begin{proof}

Given  $M_1=(a_{ij}), M_2=(b_{ij}) \in M_n(\F_q)$,  note that  $ \langle M_1\rangle \subset \langle M_2 \rangle $, and if only if, $M_1 \in \langle M_2 \rangle$, that is if there exists $(c_{ij}) \in  M_n(\F_q)$, such that $ M_1=(c_{ij})M_2.$ 

Hence, $a_{ij}=\sum_{k=1}^n c_{ik}b_{kj}$, $\forall i, j$ and thus, the $l$-th row of $M_1$ is

\begin{eqnarray*}
(a_{l1}, a_{l2}, ..., a_{ln})  &=& \left(\sum_{k=1}^n c_{lk}b_{k1}, \sum_{k=1}^n c_{lk}b_{k2}, ..., \sum_{k=1}^n c_{lk}b_{kn}\right)\\
                               %&=& \sum_{k=1}^n(c_{lk}b_{k1}, c_{lk}b_{k2}, ..., c_{lk}b_{kn})\\
                               &=& \sum_{k=1}^n c_{lk}(b_{k1}, b_{k2}, ..., b_{kn})\\
                              % &=& \sum_{k=1}^n c_{lk}L_k  
\end{eqnarray*}
 
The converse follows immediately.
\end{proof}

As an   consequence, we have 

\begin{crl}  \label{Igualsubs}$$\langle M_1 \rangle \subsetneqq \langle M_2 \rangle \Leftrightarrow  W(M_1) \subsetneqq W(M_2)$$ $$ \langle M_1 \rangle = \langle M_2 \rangle \Leftrightarrow W(M_1) = W(M_2).$$  
\end{crl}

Since every subspace of $\F_q^n$ is of the form $W(M)$, for some $M\in M_n(\F_q)$, we have the following.

\begin{thm} \label{SubspaceIdeal} Given a positive integer $k$, $1 \leq k \leq n-1$, the  map $W(N_i) \mapsto \langle N_i\rangle$ gives a bijection $\varphi$  between the set of subspaces of dimension k of  $\F_q^n$, and the set of left ideals of rank k in $M_n(\F_q)$.
\end{thm}

\begin{proof}
     Given a k-dimensional subspace $W$ of $\F_q^n$, with   basis \linebreak 
$B = \{ v_1, v_2, \cdots, v_k\}$, let us consider the matrix $M$ whose first k rows are the vectors $v_1, v_2, \cdots, v_k$ and the others are linear combinations of these ones. Clearly $W$ is equal to $W(M)$. Also if $W(M_1)= W(M_2)$, then $\langle M_1 \rangle = \langle M_2 \rangle$, by Corollary \ref{Igualsubs}. Therefore we conclude that $\varphi$ is well defined. Since, $\langle M_1 \rangle = \langle M_2 \rangle$ implies $W(M_1)= W(M_2)$, it follows that $\varphi$ is injective.   As all the ideals of $M_n(\F_q)$ are principal, $\varphi$ is also surjective.
\end{proof}

 Since   \cite[Chapter 24]{JR} gives the number of k-subespace of $\F_q^n$, with $0 < k \leq n$   we get the following. 

\begin{thm} \label{numberideals}The number of left ideals of rank k in $M_n(\F_q)$ is $$\frac{(q^n-1)(q^{n-1} -1)(q^{n-2}-1)\cdots (q^{n-k+1}-1)}{(q^k-1)(q^{k-1}-1)(q^{k-2}-1) \cdots (q-1)}.$$
%$$ \frac{(q^n-1)(q^{n-1} -1)(q^{n-2}-1)\cdots (q^{n-k+1}-1)}{(q^k-1)(q^{k-1}-1)(q^{k-2}-1) \cdots (q-1)}.$$ 
\end{thm} 

We now determine conditions on two idempotent matrices to generate the same left ideal.

\begin{thm}  \label{GeradorId} Given $R, S \in M_n(\F_q)$, the following assertions are equivalent:
 
\noindent$(i)$\;\;$R$ and $S$ are idempotents and $\langle S \rangle = \langle R \rangle$; 

\noindent$(ii)$\;$RS=R$ and $SR=S$;

\noindent$(iii)R=S+(I-S)RS$ and $S$ is an idempotent, where $I$ denotes the identity matrix of $M_n(\F_q)$.
\end{thm} 

\begin{proof}

(i) $\Rightarrow$ (ii)\\ If $R \in \langle S \rangle$ then, there exists $ A \in M_n(\F_q)$, such that, $R=AS.$
 
Hence, $RS= AS. S= AS = R$, as $S$ is idempotent. Similarly, we conclude $SR=S$.
 
(ii) $\Rightarrow$ (i)\\ Notice that (ii) implies that $R \in \langle S \rangle$ and \\$ S \in \langle R \rangle$, thus  $\langle S \rangle = \langle R \rangle$. 

Also, $S^2= (SR)S= S(RS) = SR =S $ 
 and  $R^2=(RS)R= R(SR) = RS = R.$

(ii) $\Rightarrow$ (iii)\\ We compute, $S+(I-S)RS  =  S+(I-S)R = S + R -SR  = R$.

Also $S^2 =(SR)S = S(RS) = SR =S$.

 (iii) $\Rightarrow$ (ii) \\ 
$SR= S\left(S+(I-S)RS\right)=S $, and \\  $R=S+(I-S)RS=S + RS -  SS = RS $. 

\end{proof}
 
Given an idempotent, we now determine all others idempotents that generate the same ideal.

\begin{thm} \label{Geradorg} Given an idempotent $ S \in M_n(\F_q)$   then, for all $M \in M_n(\F_q)$, the matrix $M_0 = S+(I-S)MS$ is  also an idempotent generator of $\langle S \rangle$. 

Conversely, every idempotent generator of $\langle S \rangle$ is of this form.
\end{thm}

\begin{proof}

Clearly, $M_0$ is an idempotent and, using Theorem \ref{GeradorId}, we have:
$$SM_0= S\left(S+(I-S)MS\right)=S \Rightarrow S \in \langle M_0 \rangle \; \mbox{and }$$
$$M_0S=\left(S+(I-S)MS\right)S=S+(I-S)MS=M_0  \Rightarrow M_0 \in \langle S \rangle.$$  

Therefore, $\langle S \rangle = \langle M_0 \rangle$.

The converse follows immediately from part (iii) of the previous theorem.

\end{proof}

\begin{rmk}\label{remark} {\em Note that, if we denote by $[M]$ the right ideal generated by a matrix $M\in M_n(\F_q)$ and by $V(M)$ the subspace of $\F_q^n$ generated by the columns of $M$, then the following results, duals to those in this section, hold:

\noindent (a) $[M_1]=[M_2]$ if and only of $V(M_1)=V(M_2)$.

\noindent (b) If $R,S\in M_n(\F_q)$ are idempotents,  then $[R] = [S]$ if and only if $RS=S$ and $SR=R$.}
\end{rmk}

\section{ Idempotent generators of   left ideals.} 

The purpose of this section is to exhibit  idempotent generators of   left ideals of a given rank $k$ in $M_n(\F_q)$ and to determine, for each ideal $I$, the set of all idempotent generators of $I$. We   first study the idempotents of rank 1:

\begin{thm} \label{Idem} The following matrices are idempotent generators of the different minimal left ideals of $M_n(\F _q)$ 

 $$\tiny{\left[ \begin{array}{cccccc} 1 & a_1 & a_2 & ... & a_{n-1} \\ 0 & 0 & 0 & \cdots  & 0  \\ 0 & 0 & 0 & \cdots & 0  \\ \vdots & \vdots & \vdots & \vdots &\vdots\\ 0 & 0 & 0 &\cdots & 0 \end{array} \right],  \left[ \begin{array}{cccccc} 0 & 0 & 0 & \cdots & 0 \\0 & 1 & a_2 & \cdots & a_{n-1} \\ 0 & 0 & 0 &\cdots & 0  \\  \vdots & \vdots & \vdots & \vdots  & \vdots \\ 0 & 0 & 0 & \cdots & 0  \end{array} \right], \cdots,
\left[ \begin{array}{cccccc} 0 & 0 & \cdots  & 0 & 0\\ 0 & 0 & \cdots  & 0& 0\\  \vdots & \vdots & \vdots & \vdots  \\ 0 &0 & \cdots  & 0 & 0  \\ 0 & 0 & \cdots & 0& 1  \end{array} \right]}.$$ 
% $$\left[ \begin{array}{ccccc} 0 & \cdots  & 0& 0\\ 0 & \cdots  & 0& 0\\  \vdots & \vdots & \vdots & \vdots  \\ 0 &\cdots & 1& a_{n-1}  \\ 0 & \cdots & 0& 0  \end{array} \right], 
%$$ \left[ \begin{array}{ccccc} 0 & \cdots  & 0& 0\\ 0 & \cdots  & 0& 0\\  \vdots & \vdots & \vdots & \vdots  \\ 0 &\cdots  & 0& 0  \\ 0 & \cdots & 0& 1  \end{array} \right]$$. 

 Moreover,  each minimal left ideal has $q^{n-1}$ different  idempotent generators.
\end{thm}
\begin{proof}

We observe that, according to Theorem \ref{SubspaceIdeal}, each minimal left ideal $ I = \langle M\rangle \subset M_n(\F _q)$ is related with a 1-dimensional subspace $W = W(M) \subset \F_q^n$. Let $v=(a_1, a_2, \cdots, a_n)$ be a generator of $W$. Then,  the rows of any matrix in the ideal $I$ are scalar multiples of $v$. 

Let $i$ be the smallest index such that   $a_i\neq 0$, $1 \leq i \leq n$ and set $u=(1/a_i)v$. 
Then,   $u=(0, \cdots, 0, 1, b_{i+1}, \cdots, b_n)$, is also a generator of $W$. Thus,  we can assume that $I$ is generated by  a matrix whose   $i$-th row is the vector $u$ and the others are zero. Notice that the total number of matrices of this form is $S_{n_i} =\sum_{i=1}^n q^{n-i} =  \frac{q^n-1}{q-1}$. \\

To prove the second part of the statement, let $E_{11}$ denote the matrix whose entry in position $(1,1)$ is equal to 1 and all other entries are equal to 0. Then,
$$\tiny{\langle E_{11}\rangle = \left\{ \left[\begin{array}{cccc} a_1 & 0 & \cdots & 0 \\
                                                   a_2 & 0 & \cdots & 0 \\
                                                       &   & \cdots &  \\
																									 a_n & 0 & \ldots & 0 \end{array}\right] \; | \; a_i\in \F_q, 1\leq i\leq n \right\}.}$$
																									
It is easy to see that a non-zero element in $\langle E_{11}\rangle$ is an idempotent if and only if $a_1=1$, so the number of idempotents in $\langle E_{11}\rangle$ is $q^{n-1}$ and each of them is a generator of this ideal.																						
Let $J= \langle M\rangle$ be a minimal left ideal of $M_n(\F_q)$. As $rank(M)=1$, there exist an invertible matrix $U\in M_n(\F_q)$ such that $U^{-1}MU=E_{11}$. Conjugation by $U$ induces   an automorphism $\psi_U :M_n(\F_q)\rightarrow M_n(\F_q)$ such that $\psi_U(J) = \langle E_{11}\rangle$. Consequently, $J$ also has $q^{n-1}$ generators which are idempotents.
\end{proof}

We now extend this result to left ideals of a fixed rank $k$, $1\leq k \leq n$. \\

Let $E(n,k)$ denote the set of all matrices $A=(a_{ij})$ such that there exist $k$ rows, at positions denoted $i_1, i_2, \ldots ,i_k$ verifying:

(i) Every row of $A$, except these, is a row of zeros.

(ii) $a_{i_ji_j}= 1$ and $a_{i_j,h}=0$ if $h<i_j$, $1\leq j\leq k$.

(iii) $a_{i_j,h}=0$ for $h= i_s, \; j+1\leq s  \leq k$.

\noindent The set of numbers  $i_1, i_2, \ldots ,i_k$ will be called the {\em pivotal positions} of $A$.\\

For example, $E(4,3)$ is the set of all matrices of the form:
$$\tiny{\left[\begin{array}{cccc} 1 & 0 & 0 & a_{14} \\ & 1 & 0 & a_{24} \\ & & 1 & a_{34} \\ & & & 0 \end{array}\right], \  \left[\begin{array}{cccc} 1 & 0 & a_{13} & 0 \\ & 1 & a_{23} & 0 \\ & & 0 & 0 \\ & & & 1 \end{array}\right], \ 
\left[\begin{array}{cccc} 1 & a_{12} & 0 & 0 \\ & 0 & 0 & 0 \\ & & 1 & 0 \\ & & & 1 \end{array}\right], \ 
\left[\begin{array}{cccc} 0 & 0 & 0 & 0 \\ & 1 & 0 & 0 \\ & & 1 & 0 \\ & & & 1 \end{array}\right]}$$
with each $a_{ij}\in \F_q$.\\

Clearly, every matrix in $E(n,k)$ is an idempotent, of rank $k$. We first claim that each matrix in $E(n,k)$ generates a different left ideal of $M_n(\F_q)$. 

In fact, notice that, given $R,S\in E(n,k)$, if they differ in some pivotal position, then $W(R)\neq W(S)$ so $\langle R\rangle \neq \langle S\rangle$. 

On the other hand, if $R$ and $S$ have the same pivotal positions, then it is easy to see that $[R] =[S]$ and part (b) of Remark~\ref{remark} shows that $RS=S$ and $SR=R$. If    $\langle R\rangle = \langle S\rangle$, it follows from part $(ii)$ of Theorem \ref{GeradorId} that also $RS=R$ and $SR=S$, showing that $R=S$.\\
 
We shall  show that $S(n,k)$, the number of elements in $E(n,k)$ is equal to the number of left ideals of rank $k$ in $M_n(\F_q)$.

% Notice that in a row corresponding to a pivotal position $i_j$, the first $i_j-1$ entries are equal to 0 and the entry at $i_ji_j$ is equal to 1. Also, the entries corresponding to positions $i_{j+1}i_j, \ldots ,ki_j$ are equal to 0. So, the number of entries in this row that are not fixed  is $(n-i_j -k -j)q^n$.

%Adding these values, for $1\leq j\leq k$ we see that, for a given family of pivotal positions, we have 
%$$q^{nk - \frac{(k-1)k}{2}-(i_1+i_2+\cdots+i_k)}$$
%matrices with these pivotal positions in $E(n,k)$.
%Hence, the number of matrices in $E(n,k)$ is

 %$$S_{n_k}=\sum_{1\leqslant i_1<i_2<\cdots <i_k\leqslant n} q^{nk-\frac{k(k-1)}{2}-(i_1+i_2+\cdots +i_k)}.$$

%\vspace{.5cm}We now compute this number in a different way.

\begin{lema}\label{outra}
With the notations above, we have that
$$S(n,k) = \frac{(q^n-1)(q^{n-1} -1)(q^{n-2}-1)\cdots (q^{n-k+1}-1)}{(q^k-1)(q^{k-1}-1)(q^{k-2}-1) \cdots (q-1)}.$$
\end{lema}

\begin{proof} We use induction on $n$, If $n=1$ the result is trivially true, so assume that the result holds for $n-1$ and $k\leq n-1$. 

We compute separately the number of elements in $E(n,k)$ that have the pivotal position $i_k=n$ and those which do not. Clearly, there are $S(n-1,k-1)$ idempotents that have an entry equal to 1 in position $n,n$. Since each matrix having an entry equal to 0 in position $n,n$ has $k$ entries from $\F_q$ in the last column, there exist $S(n-1,k)q^k$ such matrices, so we have:
$$S(n,k) = S(n-1,k-1) + S(n-1,k)q^{k}.$$
Using the induction hypothesis we get:

\begin{eqnarray*}
S(n,k) & = & 
\frac{(q^{n-1}-1)\cdots (q^{n-k+1}-1)}{(q^k-1)\cdots (q-1)} +\\
& & +  \frac{(q^{n-1}-1)\cdots   (q^{n-k}-1)}{(q^{k}-1)\cdots  (q-1)}\cdot q^{k} = \\
& = & \frac{(q^{n-1}-1)  \cdots (q^{n-k+1}-1)}{(q^{k-1}-1)   \cdots (q-1)} \left[\frac{ q^{n-k}-1}{q^k-1} +  q^{n-k} \right]\\
& = &  \frac{(q^{n-1}-1)  \cdots (q^{n-k+1}-1)}{(q^{k-1}-1)   \cdots (q-1)} \left[1+ \frac{(q^{n-k}-1)q^k}{q^k-1}\right]\\
& = & \frac{(q^{n-1}-1)  \cdots (q^{n-k+1}-1)}{(q^{k-1}-1)   \cdots (q-1)} . \frac{(q^n-1)}{(q^k-1)}.
\end{eqnarray*}
\end{proof}

Notice that the number of elements in $E(n,k)$ found above is equal to the number of left ideals of rank $k$ in $M_n(\F_q)$, as seen in Theorem~\ref{numberideals}.   So, we have shown the first part of the following.

\begin{thm} The elements of the set $S(n,k)$ are idempotent generators of the different left ideals of rank $k$ of $M_n(\F_q)$. Moreover, each left ideal of rank $k$ has $q^{(n-k)k}$ different idempotent generators.
\end{thm}

\begin{proof} We are left to prove only the last statement. \vspace{.3cm}

Set $F_{kk} = \mbox{diag}(\underbrace{1,  \ldots ,1}_{(k \mbox{ {\tiny times }})}, 0, \ldots 0)$. Then, using Theorem 2.6, we have that any other idempotent generator of   $\langle F_{kk}\rangle$ is of the form $$M_0 = F_{kk}+(I-F_{kk})MF_{kk}\mbox{,\; with}\; M\in M_n(\F_q).$$
$\mbox{So,}\;\tiny{M_0 = \left[\begin{array}{cc} I & 0 \\ B & 0 \end{array}\right]} \mbox{ with } I\in M_{k,k}(\F_q), \; B\in M_{n-k,k}(\F_q).$

Hence, this ideal has $q^{(n-k)k}$ generators.

Arguments similar to those in the proof of Theorem 3.1 show that all left ideals of rank $k$ have the same number of idempotent generators.   
\end{proof}

Notice that a matrix in $\langle F_{kk}\rangle$ is of rank $k$ if and only if the first $k$ rows are linearly independent. The total  number of choices of $k$ elements of $\F_q^n$ which are linearly independent is
$$(q^n-1)(q^n-q)(q^n-q^2)\cdots (q^n-q^{k-1}).$$

All left ideals of rank $k$ contain the same number of matrices of that rank, since they are isomorphic to $\langle F_{kk}\rangle$. As every matrix of rank $k$ belongs to one  and only one such an ideal, using Theorem 3.3 above, we get the following.

\begin{thm} Let $k$ be an integer, $1 \leq k\leq n$. Then, the number of matrices of rank $k$ in $M_n(\F_q)$ is
$$\frac{((q^n-1)(q^{n-1} -1)(q^{n-2}-1)\cdots (q^{n-k+1}-1))^2}{(q^k-1)(q^{k-1}-1)(q^{k-2}-1) \cdots (q-1)}.$$

\end{thm}

\end{document}